\documentclass[a4paper,11pt]{article}

\usepackage{amsmath,amsthm,amsfonts,amssymb,graphics,epsfig,color}

\usepackage[left=2cm,top=3cm,right=2cm,bottom=3cm,bindingoffset=0.5cm]{geometry}
\usepackage[latin1]{inputenc}
\usepackage[english]{babel}
\usepackage{verbatim}
\usepackage{amscd,enumerate}
\usepackage{epstopdf}
\usepackage{graphicx}
\usepackage{multicol}

\newtheorem{thm}{Theorem}[section]

\newtheorem{lem}[thm]{Lemma}
\newtheorem{rem}[thm]{Remark}

\usepackage{caption}
\usepackage{subcaption}

\usepackage{multirow} 
\usepackage{float}
\usepackage{graphicx}
\usepackage{longtable}  

\newcommand\pder[2][]{\ensuremath{\frac{\partial#1}{\partial#2}}}

\newcommand{\Ste}{\text{Ste}}
\newcommand{\Bi}{\text{Bi}}
\DeclareMathOperator\erf{erf}

\begin{document}



\title{Existence and uniqueness of solution for two one-phase Stefan problems with variable thermal coefficients}

\author{
Julieta Bollati$^{1,2}$, Mar\'ia F. Natale$^{1}$, Jos\'e A. Semitiel$^{1}$,  Domingo A. Tarzia $^{1,2}$\\ \\
\small {{$^1$} Depto. Matem\'atica, FCE, Univ. Austral, Paraguay 1950} \\  
\small {S2000FZF Rosario, Argentina.} \\
\small{{$^2$} CONICET}
}
\date{}

\maketitle

\begin{abstract}
One dimensional Stefan problems for a semi-infinite material with temperature dependent thermal coefficients are considered. Existence and uniqueness of solution are obtained imposing a Dirichlet or a Robin type condition at fixed face $x=0$. Moreover, it is  proved that the solution of the problem with the Robin type condition converges to the solution of the problem with the Dirichlet condition at the fixed face. Computational examples are provided.

\medskip


\end{abstract}

\section{Introduction.}
The one-phase Stefan problem (or Lam\'e-Clapeyron-Stefan problem) for a semi-infinite material is a free boundary problem for the heat equation, which requires the determination
of the temperature distribution $T$ of the liquid phase (melting problem) or
the solid phase (solidification problem) and the evolution of the free boundary $x = s(t)$. Phase change problems appear frequently in industrial processes and other problems of technological interest \cite{AlSo,AMR,ChRa,DHLV,Ke,Lu}.
The Lam\'e-Clapeyron-Stefan problem is non-linear even in its simplest form due to the free boundary conditions. If the thermal coefficients of the material are temperature-dependent, we have a doubly non-linear free boundary problem.
Some other models involving temperature-dependent thermal conductivity can also be found in \cite{BoNaSeTaLibro, BoNaSeTa-ThSci,NaTa,BrNa3,BrNa,Ma,ChSu74,OlSu87,Ro15,Ro18,CST}.

In this paper, we consider two one-phase fusion problems with a temperature - dependent thermal conductivity $k(T)$ and  specific heat $c(T)$. 
In one of them, it is assumed a Dirichlet condition at the fixed face $x=0$ and in the second case a Robin condition is imposed. The mathematical model of the governing process is described as follows:

\begin{align}
& \rho c(T) \pder[T]{t}=\frac{\partial}{\partial x} \left(k(T)\pder[T]{x} \right),& 0<x<s(t), \quad t>0, \label{EcCalor}\\
&  T(0,t)=T_{_0}, &t>0, \label{CondBorde}\\
&  T(s(t),t)=T_{f}, &t>0, \label{TempCambioFase}\\
&  k_0\pder[T]{x}(s(t),t)=-\rho l \dot s(t), &t>0, \label{CondStefan}\\
& s(0)=0,\label{FrontInicial}
\end{align}
where the unknown functions are the temperature $T=T(x,t)$ and the free boundary $x=s(t)$ separating both phases. The parameters $\rho>0$ (density), $l>0$ (latent heat per unit mass), $T_{0}>0$ (temperature imposed at the fixed face $x=0$) and $T_f<T_0$ (phase change temperature at the free boundary $x=s(t)$) are all known constants. The functions $k$ and $c$ are defined as:  
\begin{align}
&k(T)=k_{0}\left(1+\delta\left(\tfrac{T-T_{f}}{T_{0}-T_{f}}\right)^{p}\right)\label{k}\\
&c(T)=c_{0}\left(1+\delta\left(\tfrac{T-T_{f}}{T_{0}-T_{f}}\right)^{p}\right)\label{c},
\end{align}
where $\delta$ and $p$ are given non-negative constants, $k_{0}=k(T_f)$ and $c_{0}=c(T_f)$  are the reference thermal conductivity and the specific heat, respectively.

The problem (\ref{EcCalor})-(\ref{FrontInicial}) was firstly considered in \cite{KSR} where an equivalent ordinary differential problem was obtained. In \cite{BNT}, the existence of an explicit solution of a similarity type by using a double fixed point was given when the thermal coefficients are bounded and Lipschitz functions.

We are interested in obtaining a similarity solution to problem (\ref{EcCalor})-(\ref{FrontInicial}). More precisely, one in which the temperature $T=T(x,t)$ can be written as a function of a single variable.  Through the following change of variables:
\begin{equation}
y(\eta)=\tfrac{T(x,t)-T_{f}}{T_{0}-T_{f}}\geq 0  \label{Y}
\end{equation}
with
\begin{equation}
\eta=\tfrac{x}{2a\sqrt{t}},\quad 0<x<s(t),\quad t>0, \label{eta}
\end{equation}
the phase front moves as
\begin{equation}
s(t)=2a\lambda\sqrt{t} \label{freeboundary}
\end{equation}
where $a^{2}=\frac{k_{0}}{\rho c_{0}}$ (thermal diffusivity) and $\lambda>0$ is a positive parameter to be determined.

It is easy to see that the Stefan problem $\mathrm{(}$\ref{EcCalor}$\mathrm{)}$-(\ref{FrontInicial}) has a similarity solution $(T,s)$ given by:
\begin{align}
&T(x,t)=\left(T_{0}-T_{f}\right)y\left(\tfrac{x}{2a\sqrt{t}}\right)+T_{f},\quad  0<x<s(t), \quad t>0,\label{T} \\
&s(t)=2a\lambda\sqrt{t},\quad\quad t>0\label{s}
\end{align}
if and only if the function $y$ and the parameter $\lambda>0$ satisfy the following ordinary differential problem:
\begin{align}
&2\eta(1+\delta y^p(\eta))y'(\eta)+[(1+\delta y^p(\eta))y'(\eta)]'=0, \quad &0<\eta<\lambda, \label{y}\\
&y(0)=1,\label{cond0}\\
&y(\lambda)=0, \label{condlambda}\\
&y'(\lambda)=-\tfrac{2\lambda}{\Ste} \label{eclambda}
\end{align}
where $\delta\geq 0$, $p\geq 0$ and $\Ste=\tfrac{c_{0}(T_{0}-T_f)}{l}>0$ is the Stefan number.

In \cite{KSR},  the solution to the  ordinary differential problem (\ref{y})-(\ref{eclambda}) was approximated by using shifted Chebyshev polynomials. Although, in this paper was provided the exact solution for the particular cases $p=1$ and $p=2$, the aim of our work is to prove existence and uniqueness of solution  for every $\delta\geq 0$ and $p\geq 0$.
The particular case with $\delta=0$, i.e. with constant thermal coefficients, and $p=1$ was studied in \cite{ChSu74,OlSu87,Ta98,SaTa}

In Section 2, we are going to prove existence and uniqueness of problem (\ref{EcCalor})-(\ref{FrontInicial}) through analysing the ordinary differential problem (\ref{y})-(\ref{eclambda}).


In Section 3, we present a similar problem but with a Robin type condition at the fixed face $x=0$. That is, the temperature condition (\ref{CondBorde}) will be replaced by the following convective condition
\begin{equation}
k(T(0,t))\pder[T]{x}(0,t)=\frac{h}{\sqrt{t}}\left(T(0,t)-T_{0}\right) \label{convectiva}
\end{equation}
where $h>0$ is the thermal transfer coefficient and $T_0$ is the bulk temperature. We prove existence and uniqueness of solution to this problem, similar to those of the preceding section.

Finally, in Section 4,  we study the asymptotic
behaviour when $h\rightarrow +\infty$, that is, we show that the solution of the problem given in Section 3 converges to the solution of the analogous Stefan problem, given in Section 2.

\section{Existence and uniqueness of solution to the problem with Dirichlet condition at the fixed face $x=0$}

We will study the  existence and uniqueness of solution to the problem  (\ref{EcCalor})-(\ref{FrontInicial}) through the ordinary differential problem
(\ref{y})-(\ref{eclambda}).

\begin{lem}\label{ProbAux}
Let $p\geq 0$, $\delta\geq 0$, $\lambda>0$, $y\in C^{\infty}[0,\lambda]$ and $y\geq 0$, then $(y,\lambda)$ is a solution to the ordinary differential problem $(\ref{y})$-$(\ref{eclambda})$ if and only if $\lambda$ is the unique solution to 
\begin{align}\label{7}
f(x)=g,\qquad \qquad x>0,
\end{align}
and $y$ verifies
\begin{align}\label{6}
F(y(\eta))=G(\eta),\qquad\qquad 0<\eta<\lambda,
\end{align}
where 
\begin{align}\label{fg}
g=\tfrac{\mathrm{Ste}}{\sqrt{\pi}}\left( 1+\tfrac{\delta}{p+1}\right),  \qquad &f(x)=x \exp(x^2)\erf(x),\\
F(x)=x+\tfrac{\delta}{p+1}x^{p+1}, \qquad &G(x)=\tfrac{\sqrt{\pi}}{\mathrm{Ste}} \;\lambda \exp(\lambda^2)\left( \erf(\lambda)-\erf(x)\right).\label{FG-temp}
\end{align} 
\end{lem}

\smallskip

\begin{proof}
Let $(y,\lambda)$ be a solution to problem (\ref{y})-(\ref{eclambda}).

Let us define $v(\eta)=\left(1+\delta y^{p}(\eta) \right) y'(\eta)$. Taking into account the ordinary differential equation (\ref{y}) and condition (\ref{cond0}), $v$ can be rewritten as $v(\eta)=(1+\delta)y'(0)\exp(-\eta^2)$. Therefore
\begin{equation}\label{aux}
y'(\eta)+\delta y^p(\eta)y'(\eta)=(1+\delta)y'(0)\exp(-\eta^2).
\end{equation}

If we integrate (\ref{aux}) from $0$ to $\eta$, and using conditions (\ref{cond0})-(\ref{condlambda}) we obtain 
\begin{equation}\label{yalternativa}
y(\eta)+\tfrac{\delta}{p+1}y^{p+1}(\eta)=1+\tfrac{\delta}{p+1}-\tfrac{\sqrt{\pi}}{\mathrm{Ste}}\lambda \exp(\lambda^2)\erf(\eta)
\end{equation}

If we take $\eta=\lambda$ in the above equation, by (\ref{condlambda}), we get (\ref{7}). Furthermore, from (\ref{7}) we can rewrite (\ref{yalternativa}) as (\ref{6}).

Reciprocally, if $(y,\lambda)$ is a solution to (\ref{7})-(\ref{6})  we have
\begin{equation}
y(\eta)=-\tfrac{\delta}{p+1}y^{p+1}(\eta)+\left(1+\tfrac{\delta}{p+1}\right)\left(1-\tfrac{\erf(\eta)}{\erf(\lambda)}\right).
\end{equation}

An easy computation shows that $(y,\lambda)$  is a solution to the ordinary differential problem $\mathrm{(}$\ref{y}$\mathrm{)}$-$\mathrm{(}$\ref{eclambda}$\mathrm{)}$ .
\end{proof}

According to the above result, 
we proceed to show that there exists a unique solution to problem (\ref{7})-(\ref{6}).

\begin{lem}\label{ExyunProbAux}
If $p\geq 0$ and $\delta\geq 0$, then there exists a unique solution $(y,\lambda)$ to the problem  $\mathrm{(}$\ref{7}$\mathrm{)}$-$\mathrm{(}$\ref{6}$\mathrm{)}$ with $\lambda>0$, $y\in C^{\infty}[0,\lambda]$ and $y\geq 0$. 
\end{lem}
\begin{proof}
In virtue that $f$ given by (\ref{fg}) is an increasing function such that $f(0)=0$ and $f(+\infty)=+\infty$, there exists a unique solution $\lambda>0$ to equation $\mathrm{(}$\ref{7}$\mathrm{)}$.
Now, for this $\lambda>0$, it is easy to see that $F$ given by (\ref{FG-temp}) is an increasing function, so that we can define $F^{-1}:[0,+\infty)\to [0,+\infty)$. 
As $G$ defined by (\ref{FG-temp}) is a positive function, we have that 
there exists a unique solution $y\in C^{\infty}[0,\lambda]$ of equation $\mathrm{(}$\ref{6}$\mathrm{)}$ given by 
\begin{equation}
y(\eta)=F^{-1}\left(G(\eta)\right), \qquad  0 < \eta < \lambda.
\end{equation}

\end{proof}

\begin{rem} \label{2.3}
On one hand we have that $F$ is an increasing function with $F(0)=0$ and $F(1)=1+\frac{\delta}{p+1}$.  On the other hand, $G$ is a decreasing function with  $G(0)=1+\frac{\delta}{p+1}$ and $G(\lambda)=0$. Then it follows that $0\leq y(\eta)\leq 1 $, for $0<\eta<\lambda$.
\end{rem}

From the above lemmas we are able to claim the following result:

\begin{thm} \label{ExistenciaDirichlet}
The Stefan problem governed by $($\ref{EcCalor}$)$-$($\ref{FrontInicial}$)$ has a unique similarity type solution given by $($\ref{T}$)$-$($\ref{s}$)$ where $(y,\lambda)$ is the unique solution to the functional problem $($\ref{7}$)$-$($\ref{6}$)$.
\end{thm}

\begin{rem}
In virtue of Remark \ref{2.3} and Theorem \ref{ExistenciaDirichlet} we have that 
\begin{equation*}
T_f<T(x,t)<T_0,\qquad \qquad 0<x<s(t),\quad t>0.
\end{equation*}
\end{rem}

\begin{rem} For the particular case $p=1$, $\delta\geq 0$,
the solution to the problem (\ref{7})-(\ref{6}) is given by
\begin{align}
&y(\eta)=\tfrac{1}{\delta} \left[\sqrt{(1+\delta)^2-\delta(2+\delta)\tfrac{\erf(\eta)}{\erf(\lambda)}}-1 \right], \qquad 0<\eta<\lambda,\label{6-1}
\end{align}
where $\lambda$ verifies
\begin{align}
&\lambda \exp(\lambda^2)\erf(\lambda)=\tfrac{\mathrm{Ste}}{\sqrt{\pi}}\left( 1+\tfrac{\delta}{2}\right).\label{7-1}
\end{align}
\begin{proof}
If $p=1$ the equation (\ref{6}) is given by
\begin{equation}
y^{2}(\eta)+\tfrac{2}{\gamma}y(\eta)-(1+\tfrac{2}{\gamma})\left[1-\tfrac{\erf(\eta)}{\erf(\lambda)}\right]=0
\end{equation}
which has a unique positive solution obtained by the expression (\ref{6-1}).
\end{proof}
\end{rem}

\vspace{0.4cm}

In view of Lemmas \ref{ExyunProbAux} and \ref{2.3}, we can compute the solution $(y,\lambda)$ to the ordinary differential problem $(\ref{y})$-$(\ref{eclambda})$, by using its functional formulation.

In Figure \ref{Fig:yeta}, for different values of $p$, we plot the solution  $(y,\lambda)$  to the problem $(\ref{7})$-$(\ref{6})$. In order to compare the  obtained solution $y$,  we extend them by zero for every $\eta>\lambda$. We assume $\delta=5$ and $\Ste=0.5$. It must be pointed out that the choice for $\Ste$ is due to the fact that  for most phase-change material candidates over a realistic temperature, the Stefan number will not exceed 1 (see \cite{So}).

\begin{figure}[h]
\begin{center}
\includegraphics[scale=0.22]{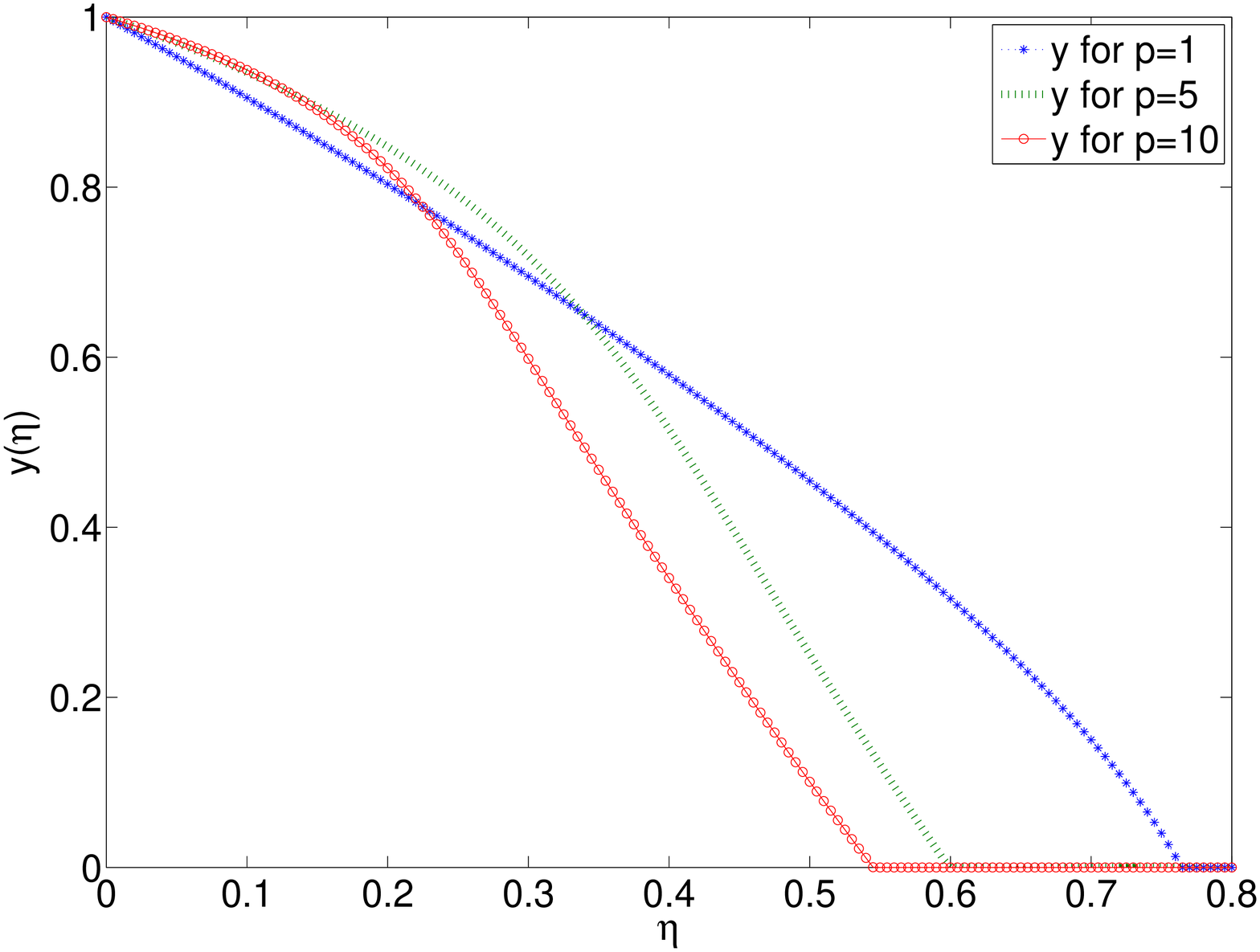}
\caption{Plot of function $y$ for different values of \mbox{$p=1,5,10$},  fixing $\delta=5$ and  $\Ste=0.5$.}
\label{Fig:yeta}
\end{center}
\end{figure}

\medskip
Although it can be analytically deduced from equation $(\ref{7})$,  we can observe graphically that as $p$ increases, the value of $\lambda$ decreases.

In view of Theorem \ref{ProbAux}, we can also plot the solution $(T,s)$ to the 
 problem  $\mathrm{(}$\ref{y}$\mathrm{)}$-$\mathrm{(}$\ref{eclambda}$\mathrm{)}$. 
 
In Figure \ref{Fig:ColorTemp} we present a colormap for the temperature $T=T(x,t)$ extending it by zero for $x>s(t)$.

\begin{figure}[h!!!!]
\begin{center}
\includegraphics[scale=0.22]{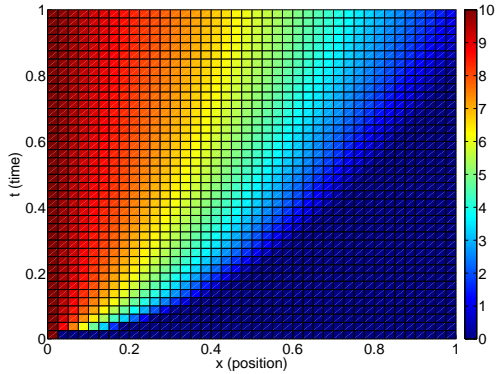}
\caption{Colormap for the temperature $T=T(x,t)$ function fixing $\delta=1$, $p=1$, $\Ste=0.5$,  $T_f=0$, $T_0=10$ and  $a=1$}
\label{Fig:ColorTemp}
\end{center}
\end{figure}

\newpage

\section{Existence and uniqueness of solution to the problem with Robin condition at the fixed face $x=0$}

In this section we are going to consider a Stefan problem with a convective boundary condition at the fixed face instead of a Dirichlet one. This heat input is the true relevant
physical condition due to the fact that it establishes that the incoming flux
at the fixed face is proportional to the difference between the temperature at
the surface of the material and the ambient temperature to be imposed.

Let us consider the free boundary problem given by $\mathrm{(}$\ref{EcCalor}$\mathrm{)}$, $\mathrm{(}$\ref{TempCambioFase}$\mathrm{)}$-$\mathrm{(}$\ref{FrontInicial}$\mathrm{)}$ and the convective condition $\mathrm{(}$\ref{convectiva}$\mathrm{)}$ instead of the temperature condition $\mathrm{(}$\ref{CondBorde}$\mathrm{)}$ at the fixed face $x=0$.

The temperature-dependent thermal conductivity $k(T)$ and the specific heat $c(T)$ are given by $\mathrm{(}$\ref{k}$\mathrm{)}$ and $\mathrm{(}$\ref{c}$\mathrm{)}$, respectively.

As in the above section, we are searching a similarity type solution. If we define   
the change of variables as $\mathrm{(}$\ref{Y}$\mathrm{)}$-$\mathrm{(}$\ref{eta}$\mathrm{)}$, the phase front moves as $\mathrm{(}$\ref{freeboundary}$\mathrm{)}$
where $a^{2}=\frac{k_{0}}{\rho c_{0}}$ (thermal diffusivity) and $\lambda_\gamma$ is a positive parameter to be determined.

It follows that $(T_\gamma,s_\gamma)$ is a solution to $\mathrm{(}$\ref{EcCalor}$\mathrm{)}$, $\mathrm{(}$\ref{TempCambioFase}$\mathrm{)}$-$\mathrm{(}$\ref{FrontInicial}$\mathrm{)}$ and $\mathrm{(}$\ref{convectiva}$\mathrm{)}$ if and only if the function $y_\gamma$ defined by (\ref{y}) and the parameter $\lambda_\gamma>0$ given by (\ref{freeboundary}) satisfy $\mathrm{(}$\ref{y}$\mathrm{)}$, $\mathrm{(}$\ref{condlambda}$\mathrm{)}$, $\mathrm{(}$\ref{eclambda}$\mathrm{)}$ and
\begin{align}
\left(1+\delta y^{p}(0)\right) y'(0)=\gamma \left(y(0)-1\right)\label{ecconvectiva}
\end{align}
where $\delta\geq 0$, $p\geq 0$, 
\begin{equation}
\gamma=2\Bi, \quad \text{and}\quad \Bi=\frac{ha }{k_{0}}
\end{equation}
 where $\Bi>0$ is the generalized Biot number.

With a few slight changes on the results obtained in the previous section, the following assertions can be established:

\begin{lem}\label{ProbAuxConv}
Let $p\geq 0$, $\delta\geq 0$, $\gamma>0$, $\lambda_\gamma>0$, $y_\gamma\in C^{\infty}[0,\lambda_\gamma]$ and $y_\gamma\geq 0$, then $(y_\gamma,\lambda_\gamma)$ is a solution to the ordinary differential problem $\mathrm{(}$\ref{y}$\mathrm{)}$, $\mathrm{(}$\ref{condlambda}$\mathrm{)}$, $\mathrm{(}$\ref{eclambda}$\mathrm{)}$ and
$\mathrm{(}$\ref{ecconvectiva}$\mathrm{)}$ if and only if $\lambda_\gamma$ is the unique  solution to the following equation
\begin{align}\label{7bis}
F(\beta_\gamma(x))=f(x),\qquad \qquad x>0,
\end{align}
and $y_\gamma$ verifies
\begin{align}
F(y_\gamma (\eta))=G_\gamma(\eta),\qquad \qquad 0< \eta<\lambda_\gamma
\label{6bis}
\end{align}
where   $f$ and $F$ are given by $(\ref{fg})$ and $(\ref{FG-temp})$, respectively and
\begin{align}\label{beta}
\beta_\gamma(x)= 1-\tfrac{2x\exp\left(x^{2}\right)}{\gamma \, \emph{\Ste}}, \qquad 0\leq x\leq \lambda_0=\beta_\gamma^{-1}(0),\\ 
G_\gamma(x)=\tfrac{\lambda_\gamma \exp\left(\lambda_\gamma^{2}\right)\sqrt{\pi}}{\emph{Ste}}\left(\erf(\lambda_\gamma)-\erf(x)\right), \quad 0<x<\lambda_\gamma.\label{Ggamma}
\end{align}
\end{lem}

\begin{proof}
Let $(y_\gamma,\lambda_\gamma)$ be a solution to problem (\ref{y}), (\ref{condlambda}), (\ref{eclambda}) and
(\ref{ecconvectiva}).

Let us define $w(\eta)=\left(1+\delta y_\gamma^{p}(\eta) \right) y_\gamma'(\eta)$. Taking into account the ordinary differential equation (\ref{y}) and the conditions (\ref{condlambda}), (\ref{ecconvectiva}), $w$ can be rewritten as $w(\eta)=y_\gamma'(\lambda_\gamma)\exp(\lambda_\gamma^2)\exp(-\eta^2)$. Therefore
\begin{equation}\label{aux-conv}
y_\gamma'(\eta)+\delta y_\gamma^p(\eta)y_\gamma'(\eta)=y'_\gamma(\lambda_\gamma)\exp(\lambda_\gamma^2)\exp(-\eta^2).
\end{equation}

If we integrate (\ref{aux}) from $\eta$ to $\lambda_\gamma$ and using conditions (\ref{condlambda}), (\ref{eclambda}) and
(\ref{ecconvectiva}) we obtain that $y_\gamma$ verifies (\ref{6bis}).

If we take $\eta=0$ in (\ref{6bis}) we  get
\begin{equation}\label{eqaux1}
y_\gamma(0)+\tfrac{\delta}{p+1}y_\gamma^{p+1}(0)=\tfrac{\sqrt{\pi}}{\mathrm{Ste}}\lambda_\gamma \exp(\lambda_\gamma^2)\erf(\lambda_\gamma).
\end{equation}

Furthermore, if we differentiate equation (\ref{6bis}) and computing this derivative at $\eta=0$ we obtain:
\begin{equation}\label{eqaux2}
y'_\gamma(0)+\delta y_\gamma^{p}(0)y'_\gamma(0)=-\tfrac{2\lambda_\gamma \exp(\lambda_\gamma^2)}{\mathrm{Ste}}
\end{equation}

From (\ref{ecconvectiva}) and (\ref{eqaux2}) we obtain
\begin{equation} \label{ygamma-0}
y_\gamma(0)=1-\tfrac{2\lambda_\gamma \exp(\lambda_\gamma^2)}
{\mathrm{Ste}}=\beta(\lambda_\gamma)\geq 0
\end{equation} and therefore (\ref{7bis}) holds.

Reciprocally, if $(y_\gamma,\lambda_\gamma)$ is a solution to (\ref{7bis})-(\ref{6bis}), an easy computation shows that $(y_\gamma,\lambda_\gamma)$ verifies (\ref{y}), (\ref{condlambda}), (\ref{eclambda}) and
(\ref{ecconvectiva}).

\end{proof}

\begin{rem}
The notations $\lambda_\gamma$ and $ y_\gamma$ are adopted in order to emphasize the dependence of the solution to problem $\mathrm{(}$\ref{y}$\mathrm{)}$, $\mathrm{(}$\ref{condlambda}$\mathrm{)}$, $\mathrm{(}$\ref{eclambda}$\mathrm{)}$ and $\mathrm{(}$\ref{ecconvectiva}$\mathrm{)}$ on $\gamma$, although it also depends on $p$ and $\delta$. This fact is going to facilitate the subsequent analysis of the asymptotic behaviour of $y_\gamma$ when $\gamma \to\infty$ $\left( h \to \infty\right)$ to be presented in Section \ref{sec_Conv}.
\end{rem}

\begin{lem}\label{ExyunProbAux-Conv}
If $p\geq 0$, $\delta\geq 0$ and $\gamma>0$, then there exists a unique solution $(y_\gamma,\lambda_\gamma)$ to the problem  $(\ref{7bis})$-$(\ref{6bis})$ with $\lambda_\gamma>0$, $y_\gamma\in C^{\infty}[0,\lambda_\gamma]$ and $y_\gamma\geq 0$. 
\end{lem}
\begin{proof}
On one hand, the function $f$ given by (\ref{fg}) is an increasing function such that $f(0)=0$ and \mbox{$f(\lambda_0)>0$} with $\lambda_0=\beta_\gamma^{-1}(0)$. 
On the other hand,  $F(\beta_\gamma)$  with $F$ given by (\ref{FG-temp}) and $\beta_\gamma$ given by (\ref{beta}), is a decreasing function for $0\leq x \leq \lambda_0$.
Notice that $F(\beta_\gamma(0))=F(1)=\tfrac{\mathrm{Ste}}{\sqrt{\pi}}\left( 1+\tfrac{\delta}{p+1}\right)$ and $F(\beta_\gamma(\lambda_0))=F(0)=0$. Therefore we can conclude that there exists a unique $0<\lambda_\gamma<\lambda_0$ that verifies (\ref{7bis}).

Now, for this $\lambda_\gamma>0$, it is easy to see that $F$ is an increasing function, so that we can define $F^{-1}:[0,+\infty)\to [0,+\infty)$. 
As $G_\gamma$ given by (\ref{Ggamma}) is a positive function, we have that 
there exists a unique solution $y\in C^{\infty}[0,\lambda_\gamma]$ of equation (\ref{6bis}) given by 
\begin{equation}
y_\gamma(\eta)=F^{-1}\left(G_\gamma(\eta)\right), \qquad  0 < \eta <\lambda_\gamma.
\end{equation}

\end{proof}

\begin{rem} \label{2.3-Conv}
On one hand we have that $F$ is an increasing function with $F(0)=0$ and $F(1)=1+\frac{\delta}{p+1}$.  On the other hand, $G_\gamma$ is a decreasing function with  $G_\gamma(0)=\lambda_\gamma \exp(\lambda_\gamma^2)\erf(\lambda_\gamma)$ and $G_\gamma(\lambda_\gamma)=0$. Then $y_\gamma$ is a decreasing function and due to (\ref{7bis}) we obtain
$$y_\gamma(0)=F^{-1}(G_\gamma(0))=\beta_\gamma(\lambda_\gamma)=1-\tfrac{2\lambda_\gamma\exp\left(\lambda_\gamma^{2}\right)}{\gamma \, \mathrm{\Ste}}<1.$$

Then it follows that $0\leq y_\gamma(\eta)\leq 1 $ for $0<\eta<\lambda_\gamma$.
\end{rem}

Finally, from the above lemmas we are able to claim the following result:

\begin{thm} \label{ExistenciaConvectiva}
The Stefan problem governed by $(\ref{EcCalor}), (\ref{TempCambioFase})$-$(\ref{FrontInicial})$ and $(\ref{convectiva})$ has a unique similarity type solution given by $($\ref{T}$)$-$($\ref{s}$)$ where $(y_\gamma,\lambda_\gamma)$ is the unique solution to the functional problem $(\ref{7bis})$-$(\ref{6bis})$.
\end{thm}

Taking into account  Lemmas \ref{ProbAuxConv} and \ref{ExyunProbAux-Conv} we compute the solution $(y_\gamma,\lambda_\gamma)$ to the ordinary differential problem  $(\ref{y})$, $(\ref{condlambda})$, $(\ref{eclambda})$ and
$(\ref{ecconvectiva})$, using its functional formulation $(\ref{7bis})$-$(\ref{6bis})$. 
Figure \ref{Fig:yetaConv1} shows the function $y_\gamma$ for a fixed $\delta=5$, $\gamma=50$, $\Ste=0.5$, varying $p=1,5,10$. 
As it was made for the problem with a Dirichlet condition at the fixed face, the solution $y_\gamma$ is extended by zero for every $\eta>\lambda_\gamma$.

\begin{figure}[h!!!]
\begin{center}
\includegraphics[scale=0.22]{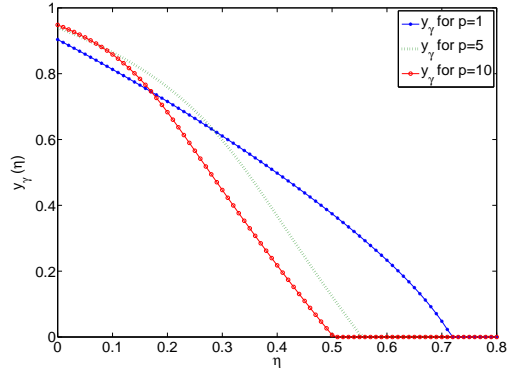}
\caption{Plot of function $y$ for different values of \mbox{$p=1,5,10$},  fixing $\delta=5$, $\gamma=50$ and  $\Ste=0.5$.}
\label{Fig:yetaConv1}
\end{center}
\end{figure}

\begin{figure}[h!!!]
\begin{center}
\includegraphics[scale=0.22]{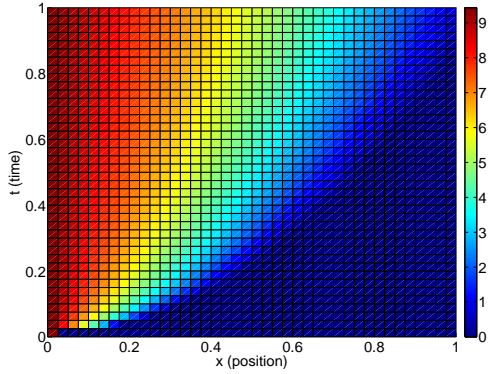}
\caption{Colormap for the temperature $T=T(x,t)$ function fixing $\delta=1$,$\gamma=50$, $p=1$, $\Ste=0.5$,  $T_f=0$, $T_0=10$ and  $a=1$}
\label{Fig:ColorConv}
\end{center}
\end{figure}

\medskip

Applying Theorem \ref{ProbAuxConv}, we can also plot the solution $(T_\gamma,s_\gamma)$ to the 
 problem  $(\ref{EcCalor}), (\ref{TempCambioFase})$-$(\ref{FrontInicial})$ and $(\ref{convectiva})$. In Figure \ref{Fig:ColorConv} we present a colormap for the temperature $T_\gamma=T_\gamma(x,t)$ extending it by zero for $x>s_\gamma(t)$.
\newpage
\section{Asymptotic behaviour}\label{sec_Conv}
Now, we will show that if  the coefficient $\gamma$, that characterizes the heat transfer at the fixed face,  goes to infinity then the solution to the problem with the Robin type condition  $(\ref{EcCalor}),(\ref{TempCambioFase})$-$(\ref{FrontInicial})$ and $(\ref{convectiva})$  converges to the solution to the problem $(\ref{EcCalor})$-$(\ref{FrontInicial})$, with a Dirichlet condition at the fixed face $x=0$.

In order to get the convergence it will be necessary to prove the following preliminary result:

\begin{lem}\label{ConvergenciaLambda}
Let $\gamma>0$, $p\geq 0$ and $\delta>0$ be. If $\lambda_\gamma$ is the unique solution to equation $(\ref{7bis})$ and $\lambda$ is the unique solution to equation $(\ref{7})$, then the sequence $\lbrace\lambda_\gamma \rbrace$  is increasing and bounded. Moreover,
$$\lim\limits_{\gamma\to\infty} \lambda_\gamma=\lambda.$$
\end{lem}

\begin{proof}
 
 Let $\gamma_1<\gamma_2$ then $F(\beta_{\gamma_1})<F(\beta_{\gamma_2})$ where $F$ is given by (\ref{FG-temp}) and $\beta_\gamma$ is defined by (\ref{beta}). Therefore $\lambda_{\gamma_1}<\lambda_{\gamma_2}$.
 In addition as  $\lim\limits_{\gamma \to\infty} F(\beta_\gamma)=g$ we have $\lambda_\gamma<\lambda$, for all $\gamma>0$.
 Finally, we obtain that $\lim\limits_{\gamma\to\infty} \lambda_\gamma=\lambda$.

\end{proof}

\begin{lem}\label{Convergenciaygamma}
Let $\gamma>0$, $p\geq 0$ and $\delta>0$ be. If $(y_{\gamma},\lambda_\gamma)$ is the unique solution to the ordinary differential problem $\mathrm{(}$\ref{y}$\mathrm{)}$, $\mathrm{(}$\ref{condlambda}$\mathrm{)}$, $\mathrm{(}$\ref{eclambda}$\mathrm{)}$, $\mathrm{(}$\ref{ecconvectiva}$\mathrm{)}$ and $(y,\lambda)$ is the unique solution to the problem $\mathrm{(}$\ref{y}$\mathrm{)}$-$\mathrm{(}$\ref{eclambda}$\mathrm{)}$, then for every $\eta\in (0,\lambda)$ the following convergence holds
\begin{equation}
\lim\limits_{\gamma \to\infty} y_\gamma(\eta)=y(\eta).
\end{equation}
\end{lem}

\begin{proof}
According to Lemmas \ref{ExyunProbAux} and \ref{ExyunProbAux-Conv} we have that $y_\gamma(\eta)=F^{-1}(G_\gamma(\eta))$, with $0<\eta<\lambda_\gamma$ and $y(\eta)=F^{-1}(G(\eta))$, with $0<\eta<\lambda$  where the functions $F$, $G$ and $G_\gamma$ are given by (\ref{FG-temp}) and 
(\ref{Ggamma}).

Let $\eta \in (0,\lambda)$. Then due to Lemma \ref{Convergenciaygamma}, there exists $\gamma_0$ such that $\eta<\lambda_\gamma$, for every $\gamma>\gamma_0$. As it can be easily seen that $G_\gamma(\eta)\to G(\eta)$ when $\gamma\to\infty$, it follows that
$$\lim\limits_{\gamma \to\infty} y_\gamma(\eta)=\lim\limits_{\gamma \to\infty} F^{-1}(G_\gamma (\eta))=F^{-1}\left( \lim\limits_{\gamma \to\infty} G_\gamma (\eta)\right)=F^{-1}(G(\eta))=y(\eta).$$
\end{proof}

In order to illustrate the results obtained in Lemmas \ref{ConvergenciaLambda} and  \ref{Convergenciaygamma}, in Figure \ref{FiguraConv}  we plot the $(y_\gamma,\lambda_\gamma)$ assuming $\delta=5$, $p=1$ and varying $\gamma=1, 25, 50,100$. We show that as $\gamma$ becomes greater, the function $y_\gamma$ converges pointwise to the solution $y$ of the problem $(\ref{y})$-$(\ref{eclambda})$.
\begin{figure}[h!!!!]
\begin{center}
\includegraphics[scale=0.22]{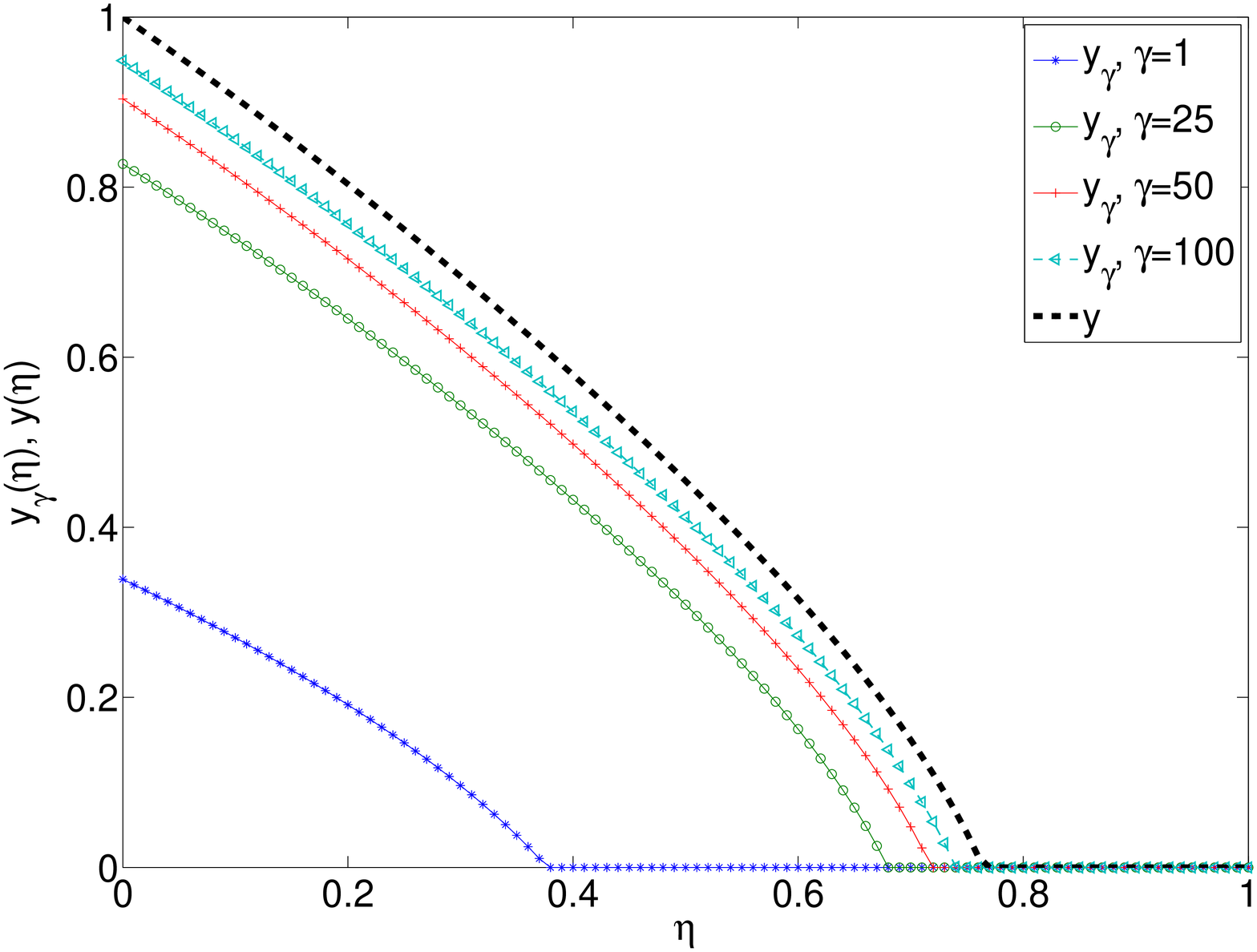}
\caption{Plot of $y_\gamma$ for $\gamma=1,25,50,100$, and $y$ functions fixing $p=1$ and $\delta=5$}
\label{FiguraConv}
\end{center}
\end{figure}

\begin{thm} \label{ConvergenciaTeo}
The unique solution $(T_\gamma,s_\gamma)$ to the Stefan problem governed by $(\ref{EcCalor}),(\ref{TempCambioFase})$-$(\ref{FrontInicial})$ and $(\ref{convectiva})$ converges pointwise to the unique solution $(T,s)$ to the Stefan problem $(\ref{EcCalor})$-$(\ref{FrontInicial})$ when $\gamma\to\infty$.
\end{thm}

\begin{proof}
The proof follows straightforward from Lemmas \ref{ConvergenciaLambda}, \ref{Convergenciaygamma} and formulas (\ref{T})-(\ref{s}).
\end{proof}

\section{Conclusions}
One dimensional Stefan problems with temperature dependent thermal coefficients and a Dirichlet or a Robin type condition at fixed face $x=0$ for a semi-infinite material were considered. Existence and uniqueness of solution was obtained in both cases. Moreover, it was  proved that the solution of the problem with the Robin type condition converges to the solution of the problem with the Dirichlet condition at the fixed face. For a particular case, an explicit solution was also obtained. In addition, computational examples were provided in order to show the previous theoretical results.

\section*{Acknowledgement}
The present work has been partially sponsored by the Project PIP No 0275 from
CONICET-UA, Rosario, Argentina, and ANPCyT PICTO Austral 2016 No 0090.


\begin{thebibliography}{1}

\bibitem{AlSo} V. Alexiades,  A.D. Solomon, Mathematical Modelling of Melting and Freezing Processes, Hemisphere-Taylor: Francis, Washington, 1993.


\bibitem{AMR} I. Athanasopoulos,  G. Makrakis, J.F. Rodrigues (Eds.), 
Free Boundary Problems: Theory and Applications, CRC Press, Boca Raton, 1999.


\bibitem{ChRa} J.M. Chadam, H. Rasmussen  (Eds.), Free boundary
problems involving solids, Pitman Research Notes in Mathematics
Series 281, Longman, Essex, 1993.

\bibitem{DHLV} J.I. Diaz, M.A. Herrero,  A. Li\~{n}an, J.L. Vazquez 
(Eds.), Free boundary problems: theory and applications, Pitman Research Notes in Mathematics Series 323, Longman, Essex, 1995.


\bibitem{Ke} N.  Kenmochi (Ed.), Free Boundary Problems: Theory and Applications, I,II., Gakuto International Series: Mathematical
Sciences and Applications, Gakkotosho,
Tokyo, 2000.


\bibitem{Lu} V.J. Lunardini, Heat transfer with freezing and thawing, Elsevier, Amsterdam, 1991.


\bibitem{BoNaSeTaLibro} J. Bollati,  M.F. Natale, J.A. Semitiel,  D.A. Tarzia,  Approximate solutions to the one-phase Stefan problem with non-linear temperature-dependent thermal conductivity in: J. Hristov- R. 
Bennacer (Eds.), Heat Conduction: Methods, Applications and Research,  Nova Science Publishers, Inc., 2019, pp.1-20.


\bibitem{BoNaSeTa-ThSci} J. Bollati,  M.F. Natale, J.A. Semitiel, D.A Tarzia, 
Integral balance methods applied to non-classical Stefan problems, Thermal Science (2018), In press.

\bibitem{NaTa} M.F. Natale, D.A. Tarzia, Explicit soltuions to the two-phase Stefan problem for Storm-type materials, J. Phys. A: Math. Gen.  33 (2000) 395-404.


\bibitem{BrNa3} A.C. Briozzo, M.F. Natale, One-dimensional nonlinear Stefan problems in Storm's materials, Mathematics, Special Issue on Partial Differential Equations 2 (2014) 1-11.


\bibitem{BrNa} A.C. Briozzo,  M.F. Natale, Nonlinear Stefan problem with convective boundary condition in Storm's materials, Z. Angew. Math. Phys.  67(2) (2016) 1-11.


\bibitem{Ma} O.D. Makinde, N. Sandeep, T.M. Ajayi,  I. L. Animasaun, Numerical Exploration of Heat Transfer and Lorentz Force Effects on the Flow of MHD Casson Fluid over an Upper Horizontal Surface of a Thermally Stratified Melting Surface of a Paraboloid of Revolution, Int. J. Nonlinear Sci. Simul. 19(2-3) (2018) 93-106.
 
 
 
\bibitem{ChSu74}  S.H. Cho, J.E. Sunderland, Phase-change problems with temperature-dependent thermal conductivity, J. Heat Transfer  96 (1974) 214-217.
 
 
 \bibitem{OlSu87}  D.L.R. Oliver , J.E. Sunderland , A phase-change problem with temperature-dependent thermal conductivity and specific heat, Int. J. Heat Mass Transfer 30 (1987) 2657-2661.
 
 
 \bibitem{Ro15}  C. Rogers, On a class of reciprocal Stefan moving boundary problems, Z. Angew. Math. Phys. 66 (2015) 2069-2079.
 
 
 
  \bibitem{Ro18}  C. Rogers, On Stefan-type moving boundary problems with heterogeneity: canonical reduction via conjugation of reciprocal transformation,  Acta Mech. (2018) https://doi.org/10.1007/s00707-018-2329-6.

\bibitem{CST}   A.N. Ceretani,  N.N. Salva, D.A.  Tarzia, An exact solution to a Stefan problem with variable thermal conductivity and a Robin boundary condition, Nonlinear Analysis: Real World Applications 40 (2018) 243-259.  
  
  
\bibitem{KSR} A. Kumar, A.K. Singh, Rajeev, A moving boundary problem with variable specific heat and thermal conductivity, Journal of King Saud University - Science (2018)  https://doi.org/10.1016/j.jksus.2018.05.028.


\bibitem{BNT}  A.C. Briozzo, M.F. Natale, D.A. Tarzia, Existence for an exact solution for a one-phase Stefan problem with nonlinear thermal coefficients from Tirskii's method, Nonlinear Analysis 67 (2007) 1989-1998.

\bibitem{Ta98} D.A. Tarzia, The determination of unknown thermal coefficients through phase change process with temperature-dependent thermal conductivity 25 (1998) 139-147.

\bibitem{SaTa}  N.N. Salva, D.A. Tarzia, A sensitivity analysis for the determination of unknown thermal coefficients through a phase-change process with temperature-dependent thermal conductivity, International Communications in Heat and Mass Transfer 38 (2011) 418-424.

\bibitem{So} A.D. Solomon, An easily computable solution to a two-phase Stefan problem, Solar energy 33 (1979) 525-528.  



\end{thebibliography}
\end{document}